\documentclass[12pt]{article}
\usepackage[utf8]{inputenc}
\usepackage{amsthm,amssymb, amsmath, color}
\usepackage{amsfonts}

\newcommand{\Z}{\mathbb{Z}}

\newcommand{\G}{\mathbb{Z}[i]}
\newcommand{\Geven}{\mathbb{Z}[2i]}

\DeclareMathOperator{\imag}{Im}
\DeclareMathOperator{\real}{Re}
\DeclareMathOperator{\sgn}{sgn}

\title{Gaussian Happy Numbers}
\author{
Breeanne Baker Swart \\
Department of Mathematical Sciences \\
The Citadel\\
171 Moultrie St.,
Charleston, SC 29409 \\
{\tt breeanne.swart@citadel.edu} \\
\ \\
Susan Crook\\
Division of Mathematics, Engineering and Computer Science\\
Loras College\\ 
1450 Alta Vista St.,
Dubuque, IA 52001\\
{\tt susan.crook@loras.edu} \\
\ \\
Helen G. Grundman\\
Department of Mathematics\\
Bryn Mawr College\\
101 N. Merion Ave.,
Bryn Mawr, PA 19010\\
{\tt grundman@brynmawr.edu} \\
\ \\
Laura L. Hall-Seelig\\
Department of Mathematics\\
Merrimack College\\
315 Turnpike Street,
North Andover, MA 01845\\
{\tt hallseeligl@merrimack.edu} 
}

\begin{document}
\theoremstyle{plain}
\newtheorem{theorem}{Theorem}
\newtheorem{cor}[theorem]{Corollary}
\newtheorem{lemma}[theorem]{Lemma}
\newtheorem{proposition}[theorem]{Proposition}
\newtheorem{claim}[theorem]{Claim}
\newtheorem{quest}[theorem]{Question}
\newtheorem{conjecture}[theorem]{Conjecture}

\theoremstyle{definition}
\newtheorem{definition}{Definition}

\everymath{\displaystyle}

\maketitle

\begin{abstract}
This paper extends the concept of a $B$-happy number, for $B \geq 2$, from the rational integers, $\Z$, to the Gaussian integers, $\G$.  
We investigate the fixed points and cycles of the Gaussian
$B$-happy functions, determining them for small values of $B$ and providing a method for computing them for any $B \geq 2$.  We discuss heights of Gaussian $B$-happy numbers, proving results concerning the smallest Gaussian $B$-happy numbers of certain heights.  Finally, we prove conditions for the existence and non-existence of arbitrarily long arithmetic sequences of Gaussian $B$-happy numbers.
\end{abstract}

\section{Introduction}
Happy numbers~\cite{oeis} and, for bases other than 10, generalized happy numbers~\cite{genhappy}, are defined in terms of iterating the base $B$ happy function $S_B:\Z^+ \rightarrow \Z^+$, defined by
\begin{equation*}
S_B\left(\sum_{j=0}^n a_j B^j  \right) =
\sum_{j=0}^n a_j^2,
\end{equation*}
where $B\geq 2$, $a_n \neq 0$, and, for each $j$, $0\leq a_j \leq B-1$.
The function has been generalized in other ways, allowing for other exponents~\cite{genhappy,fifth}, and allowing for the addition of an augmentation constant~\cite{augment, desert, oasis}.  In this paper, we extend the concept of generalized happy numbers to $\G$, the set of Gaussian integers.  Although we restrict our attention to the case with exponent two, we note that higher exponents may also lead to interesting results.

Let $B \geq 2$. For $a + b i \in \G - \{0\}$, we write 
\begin{equation}\label{notation}
a + b i = \sum_{j = 0}^n (a_j + b_j i) B^j,
\end{equation}
where $a_n$ and $b_n$ are not both $0$ and, for each $j$,
$\sgn(a) a_j \geq 0$, $\sgn(b) b_j \geq 0$, 
$|a_j| \leq B - 1$, and $|b_j| \leq B - 1$. 
Note that these conditions mean that each nonzero $a_j$ has the same sign as $a$ and each nonzero $b_j$ has the same sign as $b$.


\begin{definition}
For an integer $B\geq 2$, the \emph{Gaussian $B$-happy function} $S_B : \G \to \G$ is defined by $S_B (0) = 0$ and for $a + b i \in \G - \{0\}$,
\[
S_B (a + b i) = \sum_{j = 0}^n \left(a_j + b_j i\right)^2 = \sum_{j = 0}^n \left(a_j^2 - b_j^2\right) + 2 \left(\sum_{j = 0}^n a_jb_j\right)i.
\]
A Gaussian integer $a + b i$ is a \emph{(Gaussian) $B$-happy number} if, for some $k \in \Z^+,$ $S_B^k (a + b i) = 1$. 
\end{definition}

We note that the Gaussian $B$-happy function, when restricted to rational integers, agrees with the generalized $B$-happy function. Hence, no confusion should result from expanding the definition of the notation $S_B$ and of the term $B$-happy number in this way.  For clarity, at times we use the term \emph{rational $B$-happy numbers} to differentiate them from the Gaussian $B$-happy numbers.

We begin with some basic properties of the  Gaussian $B$-happy function. Each is proved by a straight-forward calculation.
 
\begin{lemma}\label{basic}
The following hold for each $a + b i \in \G$.
\begin{enumerate}
    \item \label{neg}
    $S_B (a + b i) =  S_B ( - ( a + b i))$.
    \item \label{conj}
    $S_B (\overline{a + b i})  =  \overline{S_B (a + b i)}$.
    \item \label{multbyi}
    $S_B \left(i(a + b i)\right) = -S_B ( a + b i)$.
    \item \label{real}
    $S_B (a + b i) \in \mathbb{R}$ if and only if for every $j$, $a_j b_j = 0$.
    \item \label{pimag}
    $S_B (a + b i)$ is purely imaginary if and only if $a = \pm b$.
    \item \label{2i}
    $S_B (a + b i)\in \Geven$, i.e., $\imag (S_B (a + b i))$ is even.
\end{enumerate}
\end{lemma}

The following is immediate from Lemma~\ref{basic}(\ref{neg}),(\ref{conj}), and (\ref{multbyi}).

\begin{lemma}\label{rep}
Fix $B\geq 2$.  If $z\in \G$ is a Gaussian $B$-happy number, then so are $-z$, $\pm iz$, $\pm \overline{z}$, and $\pm i\overline{z}$.
\end{lemma}




Although $S_B$ is not an additive function, it has a useful additive property, which generalizes directly to the Gaussian case.

\begin{lemma}\label{pullapart}
Let $a$, $b$, $c$, $d$, $r \in \Z_{\geq 0}$.  If $B^r > \max\{c,d\}$, then
\[S_B\left((a + bi)B^r + (c + di)\right) = S_B(a + bi) + S_B(c + di).\]
\end{lemma}

\begin{proof}
Since $B^r > \max \{c, d\},$ there exist rational integers $c_j$ and $d_j$ such that 
\[c + d i = \sum_{j=0}^{r - 1} (c_j + d_j i)B^j,\] with $0 \leq c_j \leq B-1 $ and $0 \leq d_j \leq B-1$. Using the usual notation, as in~(\ref{notation}), for $a + bi$, we have
\[(a + bi) B^r = \sum_{j=0}^n (a_j + b_j i)B^{j+r} = \sum_{j=r}^{n+r} (a_{j-r} + b_{j-r} i)B^{j}.\]  Thus
\begin{align*}
S_B ((a + bi) B^r + (c + d i))  
&= S_B\left(\sum_{j=r}^{n+r} (a_{j-r} + b_{j-r} i)B^{j} + \sum_{j=0}^{r - 1} (c_j + d_j i)B^j\right) \\
&= \sum_{j = r}^{n+r} (a_{j-r} + b_{j-r} i)^2 + \sum_{j = 0}^{r-1} (c_j + d_j i)^2 \\
&= \sum_{j = 0}^{n} (a_{j} + b_{j} i)^2 + \sum_{j = 0}^{r-1} (c_j + d_j i)^2 \\
&= S_B (a + bi) + S_B (c + d i),
\end{align*}
as desired.
\end{proof}

The remainder of this paper is organized as follows.
In Section 2, we provide a method for computing the fixed points and cycles for the Gaussian
$B$-happy functions and apply it to $S_B$ for $2 \leq B \leq 10$.  In Section 3, we consider heights of Gaussian $B$-happy numbers. Finally, in Section 4, we discuss the existence and non-existence of arbitrarily long arithmetic sequences of Gaussian $B$-happy numbers.

\section{Fixed Points and Cycles of $S_B$}

In this section, we examine the trajectories of the function $S_B$, identifying all fixed points and cycles of the functions, for $2\leq B \leq 10$. First, we prove that, for each $B \geq 2$, when a Gaussian integer is ``sufficiently large," the output of the Gaussian $B$-happy function has a smaller absolute value than the input. This allows for a computer search leading to Tables~\ref{fixedbegin} and~\ref{fixedend}.  Note that by Lemma~\ref{basic}(\ref{conj}) nonreal fixed points and cycles come in conjugate pairs.

\begin{theorem}\label{threedigits}
Let $g \in \G$ satisfy
\[\max\left\{|\real(g)|,|\imag(g)|\right\} \geq \begin{cases}
B^3,&\text{if } B\geq 7;\\
B^4,&\text{if } 2\leq B \leq 6.
\end{cases}
\]
Then $|S_B (g)| < |g|$.
\end{theorem}

\begin{proof}
Fix $B \geq 2$.
Let $g = a + b i = \sum_{j = 0}^n (a_j + b_j i) B^j$, as in~(\ref{notation}), with $n \geq 3$ for all $B \geq 2$, and with the added condition that $n \geq 4$ if $B\leq 6$.

For each $0\leq j\leq n$, since $|a_j| \leq B - 1$ and $|b_j| \leq B - 1$, we have
$|a_j^2 - b_j^2| \leq (B - 1)^2$ and $|a_j b_j| \leq (B - 1)^2$.
Thus
\begin{align*}
|S_B(g)| & = \left|\sum_{j = 0}^n (a_j^2 - b_j^2) + 2 \left(\sum_{j = 0}^n a_jb_j\right)i\right|\\ & =
\sqrt{\left(\sum_{j = 0}^n (a_j^2 - b_j^2)\right)^2 + \left(2 \sum_{j = 0}^n a_jb_j\right)^2}\\ & \leq
\sqrt{((B - 1)^2(n+1))^2 + (2(B - 1)^2(n+1))^2} \\ & = 
\sqrt{5}(n+1)(B - 1)^2.
\end{align*}

On the other hand, $|g| \geq \max\{|\real(g)|,|\imag(g)|\} \geq B^n$.  So it suffices to prove that, regardless of the value of $n$, $B^n > \sqrt{5}(n+1)(B - 1)^2$.

The inequality is easy to verify for $2 \leq B \leq 6$ with $n = 4$, and for $B = 7$ or $8$ with $n = 3$.  For $B \geq 9$ with $n = 3$, note that $B > 4 \sqrt{5}$,  which implies that $B^n > \sqrt{5} (n + 1) (B - 1)^2$.
Proceeding now by induction on $n$, fix $B \geq 2$ and assume that $B^n > \sqrt{5} (n + 1) (B - 1)^2$.  It follows that $B^{n+1} > B\sqrt{5}(n+1)(B - 1)^2 > \sqrt{5}(n+2)(B - 1)^2$, as desired.




Hence,
$|S_B(g)| \leq \sqrt{5}(n+1)(B - 1)^2 < B^n \leq |g|$.
\end{proof}

The following corollary is immediate.

\begin{cor}\label{allcyclesthm}
Let $B \geq 2$. Every cycle of $S_B$ contains a point $g = a + bi$ such that $|a| < B^n$ and $|b| < B^n$, where $n = 4$ if $2 \leq B \leq 6$ and $n=3$ if $B\geq 7$.
In particular, we have that every fixed point $g = a + bi$ of $S$ satisfies $|a| < B^n$ and $|b| < B^n$.
\end{cor}

It follows from Corollary~\ref{allcyclesthm} that a direct computer search can determine all fixed points and cycles of the function $S_B$, for any fixed $B \geq 2$.  For $2 \leq B \leq 6$, we applied $S_B$ iteratively to each value $0\leq g < B^4$, recording the resulting fixed points and cycles in Table~\ref{fixedbegin}.  For $7 \leq B \leq 10$, we applied $S_B$ iteratively to each value $0\leq g < B^3$, recording the resulting fixed points and cycles in Table~\ref{fixedend}.  The programs were run using each of MATLAB and Mathematica, thus proving Theorem~\ref{fixed}.

\begin{theorem}\label{fixed}
For $2 \leq B\leq 10$, the fixed points and cycles of $S_B$ are as given in Tables~\ref{fixedbegin} and~\ref{fixedend}.
\end{theorem}

\begin{table}[tbh]
\begin{center}
\begin{tabular}{|c|l|}\hline
$B$ & Fixed Points and Cycles, expressed in base $B$\\
\hline \hline
2 & 0, 1 \\ \hline
3 & 0, 1, 12, 22, 2+11i, 2-11i,\\
& 2 $\rightarrow$ 11 $\rightarrow$ 2, \\ 
&$-1+2i\to -10-11i\to -1+2i$ (and its conjugate), \\
&$-11+22i\to -20-22i\to -11+22i$ (and its conjugate)\\
\hline
4 & 0, 1 \\ \hline
5 & 0, 1, 23, 33, \\
& 4 $\rightarrow$ 31 $\rightarrow$ 20 $\rightarrow$ 4, \\ 
&$3+11i\to 12+11i\to 3+11i$ (and its conjugate) \\
\hline
6 & 0, 1, \\
& 5 $\rightarrow$ 41 $\rightarrow$ 25
$\rightarrow$  45
$\rightarrow$ 105 $\rightarrow$ 42$\rightarrow$ 32
$\rightarrow$ 21 $\rightarrow$ 5\\
\hline
\end{tabular}
\caption{Fixed points and cycles of $S_B$, $2 \leq B \leq 6$.}
\label{fixedbegin}
\end{center}
\end{table}

\begin{table}[bh]
\begin{center}
\begin{tabular}{|c|l|}\hline
$B$ & Fixed Points and Cycles, expressed in base $B$\\
\hline \hline
7 & 0, 1, 13, 34, 44, 63, 25+31i, 25-31i,\\
& 2 $\rightarrow$ 4 $\rightarrow$ 22 $\rightarrow$ 11 $\rightarrow$ 2,
16 $\rightarrow$ 52 $\rightarrow$ 41  $\rightarrow$ 23 $\rightarrow$ 16,\\
&$-15+116i\to -15-116i\to-15+116i$,\\
&$-31+44i\to -31-44i\to-31+44i$,\\
&$-11+51i\to -33-15i\to -11+51i$ (and its conjugate),\\
&$-21+26i\to-50-26i\to -21+26i$ (and its conjugate),\\
&$-1+13i\to -12 - 6 i\to -43 + 33 i\to 10 - 60 i\to$ \\
&\hphantom{mmm}$-50 - 15 i\to -1 + 13 i$ (and its conjugate),\\
&$4+22i\to 11+22i\to -6+11i\to 46-15i\to 35-125i\to $\\
&\hphantom{mmm}$4-116i\to -31-66i\to -116+66i\to -46-150i \to $ \\
&\hphantom{mmm}$ 35+55i\to -22+143i\to -24-40i\to 4+22i$ \\
&\hphantom{mmm}(and its conjugate),\\
&$14+35i\to -23 + 64i\to -54 - 66 i\to -43 + 213i\to 14 - 35i\to $\\
&\hphantom{mmm} $-23 - 64i\to -54 + 66 i\to -43 - 213 i\to 14+35i$,\\
&$-13+15i\to -22 - 44 i\to -33 + 44 i\to -20 - 66 i\to $\\
&\hphantom{mmm}$-125 + 33 i\to 15 - 60 i \to -13-15i\to -22 + 44 i\to $\\
&\hphantom{mmm}$-33 - 44 i\to -20 + 66 i\to -125 - 33 i\to 15 + 60 i \to -13+15i$\\
\hline
8 &$ 0, 1, 24, 64,15+32i, 15-32i, 45+20i, 45-20i$,\\
& 4 $\rightarrow$ 20 $\rightarrow$ 4,
15 $\rightarrow$ 32 $\rightarrow$ 15,
5 $\rightarrow$ 31 $\rightarrow$ 12 $\rightarrow$ 5,\\
&$-34+72i\to -34-72i\to -34+72i$, \\
&$-11+24i\to -22-14i\to -11+24i$ (and its conjugate) \\
&$-40+70i\to -41-70i\to 40+70i$ (and its conjugate), \\
&$4+6i\to -24+60i\to -20-30i\to -5+14i\to 10-50i\to $\\
&\hphantom{mmmm}$-30-12i\to 4+6i$ (and its conjugate) \\ 
\hline
9 & 0, 1, 45, 55, \\
& $75\to 82 \to 75 $,
58 $\rightarrow$ 108 $\rightarrow$ 72$\rightarrow$ 58, \\
&$-4+26i\to -26-53i\to 6+62i\to -4+26i  $ (and its conjugate), \\
&$10+26i\to -43+4i\to 10-26i\to -43-4i\to 10+26i $ \\
\hline
10 & 0, 1, \\
& $4 \rightarrow 16 \rightarrow 37 \rightarrow 58 \rightarrow 89
\rightarrow 145
\rightarrow 42 \rightarrow 20$ $\rightarrow$ 4,\\
& $-52+90i \to -52-90i \to -52+90i$,\\ 
&$35+48i \to -46+104i \to 35-48i \to -46-104i\to  35+48i$, \\ 
&$-15+90i \to -55-18i \to -15+90i$  (and its conjugate) \\
\hline
\end{tabular}
\caption{Fixed points and cycles of $S_B$, $7 \leq B \leq 10$.}
\label{fixedend}
\end{center}
\end{table}

Looking at the odd bases in Tables~\ref{fixedbegin} and~\ref{fixedend}, notice that $S_3$ has fixed points $12_{(3)}$ and $22_{(3)}$, $S_5$ has fixed points $23_{(5)}$ and $33_{(5)}$, $S_7$ has $34_{(7)}$ and $44_{(7)}$, and $S_9$ has $45_{(9)}$ and $55_{(9)}$.  We prove that this pattern holds for all odd bases.

\begin{theorem}\label{fixedeg}
For $B\geq 3$ odd, the numbers 
$(B^2 + 1)/2$ and $(B + 1)^2/2$
are each fixed points of the function $S_B$.
\end{theorem}

\begin{proof}
Writing each of these in base $B$ notation, we have
\[\frac{B^2 + 1}{2} = \left(\frac{B-1}{2}\right)B+\frac{B+1}{2}
\mbox{\ and\ }
\frac{(B + 1)^2}{2} = \left(\frac{B+1}{2}\right)B+\frac{B+1}{2}.\]
\clearpage
\noindent
Direct calculation then yields
\begin{align*}
S_B\left(\frac{B^2+1}{2}\right) 
    &= \left(\frac{B-1}{2}\right)^2+\left(\frac{B+1}{2}\right)^2 
    = \frac{B^2+1}{2}
\end{align*}
and
\begin{align*}
S_B\left(\frac{(B + 1)^2}{2}\right)     &= \left(\frac{B+1}{2}\right)^2+\left(\frac{B+1}{2}\right)^2 
 = \frac{(B + 1)^2}{2}.
\end{align*}
Thus, for odd $B$, $(B^2 + 1)/2$ and $(B + 1)^2/2$
are fixed points of $S_B$.
\end{proof}

\section{Heights of Gaussian Happy Numbers}
As defined in~\cite{heights}, the height of a $B$-happy number, $a$, is the smallest $k\in \Z_{\geq 0}$ such that $S_B^k(a) = 1$.  The smallest (rational) happy numbers of heights up to at least 12 are known~\cite{onheights,heights}.  
In this section, we first determine the smallest Gaussian $B$-happy numbers of heights 0, 1 and 2, showing that the results are independent of the value of $B$. We then find the smallest Gaussian $B$-happy numbers of height three for $2\leq B \leq 10$.  
Finally, we describe how to find the smallest Gaussian (10-)happy numbers of various heights and compute them for heights less than seven.

Here ``smallest" is taken to mean ``smallest absolute value."  We note that this means that the smallest number of a given height is, generally, not unique.  In fact, for $z\in \G$ of height above two, it follows from Lemma~\ref{basic} that the height and absolute value of $z$ are the same as those of $-z$, $\pm \overline{z}$, $\pm iz$, and $\pm i\overline{z}$. In results for these heights, we record representative numbers, noting that they stand for an entire equivalence class, as described in Lemma~\ref{rep}.

We first show that the values of the smallest Gaussian $B$-happy numbers of heights 0, 1, and 2 are independent of the value of $B$.
\begin{theorem}
Let $B \geq 2$.  The smallest Gaussian $B$-happy numbers of heights 0 and 1 are 1 and -1, respectively.  The smallest Gaussian $B$-happy numbers of height 2 are $i$ and $-i$.
\end{theorem}

\begin{proof}
Since 0 is not a $B$-happy number, the smallest $B$-happy numbers must be of absolute value 1 and, hence, in the set $\{1,-1,i,-i\}$.  Each of these is a Gaussian $B$-happy number and so is the smallest of its height.
\end{proof}

For height 3 and above, the base is significant. 
As seen in Table~\ref{height3}, the smallest Gaussian $B$-happy numbers of height 3 are the same for bases 2 and 4, and those in base 8 and 10 are integer multiples of those for base 2 and 4.  The smallest numbers for the other small bases do not appear to follow a pattern.  The following theorem is verified by direct calculation.

\begin{theorem}
The smallest Gaussian $B$-happy numbers of height 3 for $2 \leq B \leq 10$, are the values, $z$, given in Table~\ref{height3}, along with $-z$, $\pm \overline{z}$, $\pm iz$, and $\pm i\overline{z}$.
\end{theorem}

\begin{table}[hbt]
\begin{center}
\begin{tabular}{|c|c|}
\hline
Base  & Smallest Height 3 \\
\hline
2&$1+i$\\
3&$7+10i$\\
4&$1+i$\\
5&$4+15i$\\
6&$11+17i$\\
7&$20+27i$\\
8&$2+2i$\\
9&$2+9i$\\
10&$12+12i$\\
\hline
\end{tabular}
\caption{Representative Smallest Height 3 Gaussian $B$-Happy Numbers.}
\label{height3}
\end{center}
\end{table}

Focusing now on base 10, the smallest Gaussian happy numbers of heights up to a given fixed height can be found using a direct search by computer, or even by hand.  Noting that in the (positive) rational integers, the smallest happy numbers of heights less than 6 are all less than or equal to 23, finding the heights of all Gaussian happy numbers of absolute value at most 23, and then identifying the smallest one of each height, necessarily identifies the smallest ones of heights less than 6.  This search, in fact, identifies the smallest numbers of all heights less than 7, as presented in Table~\ref{heights}.  

\begin{theorem}
The smallest Gaussian happy numbers of heights 0 through 2 are given in Table~\ref{heights}. The smallest Gaussian happy numbers of heights 3 through 6 are the numbers $z$, given in Table~\ref{heights}, along with $-z$, $\pm \overline{z}$, $\pm iz$, and $\pm i\overline{z}$.
\end{theorem}

\begin{table}[hbt]
\begin{center}
\begin{tabular}{|r||c|c|c|c|c|c|c|} 
\hline
Height&0&1&2&3&4&5&6\\ 
\hline
Happy&1&10&13&23&19&7&365\\
\hline
Gaussian Happy&1&-1&$\pm i$&$12+12i$&$ 4+4i$&$7$&$5+19i$
\\
\hline
\end{tabular}
\caption{Smallest Happy Numbers~\cite{heights} and Representative Gaussian Happy Numbers of Small Heights.}
\label{heights}
\end{center}
\end{table}

\section{Arithmetic Sequences}

We now consider arithmetic sequences of Gaussian $B$-happy numbers. Following convention, for $D \in \G - \{0\}$, a
{\em $D$-consecutive sequence} is an arithmetic sequence with constant difference $D$.
El-Sedy and Siksek~\cite{basetenseq} showed that there exist arbitrarily long finite 1-consecutive sequences of rational (base 10) happy numbers.  Independently, Grundman and Teeple~\cite{consec} proved the more general result, given below. They also proved that the constant differences given in Theorem~\ref{GT} are the best possible.

 
\begin{theorem}[Grundman \& Teeple]
\label{GT}
If $B \geq 2$ and \[d = \gcd(2,B-1),\]
then there exist arbitrarily long finite
$d$-consecutive sequences of $B$-happy numbers. 
\end{theorem}

In this section, we prove, for various values of $D$, that there exist arbitrarily long finite $D$-consecutive sequences of Gaussian $B$-happy numbers, depending on the parity of $B\geq 2$.  We begin by showing that when the base $B$ is odd, such a $D$ must be a $\G$-multiple of $1 + i$.  It follows that all Gaussian $B$-happy numbers are contained in a single coset of the ideal $(1 + i)\G$.  This is the Gaussian analogy to the fact that for $B$ odd, all rational $B$-happy numbers are odd.

\begin{theorem}\label{coset}
Let $B \geq 3$ be odd.  Each Gaussian $B$-happy number is an element of $1 + (1+i)\G$.  In particular, if there is a $D$-consecutive sequence of (at least two) Gaussian $B$-happy numbers, then $D \in (1+i)\G$.
\end{theorem}

\begin{proof}
Assume that $B \geq 3$ is odd and note that, since $2 \in (1 + i)\G$, for any $a + bi \in \G$,
\begin{align*}
S_B(a + bi) &= S_B\left(\sum_j (a_j + b_ji)B^j \right) \equiv \sum_j (a_j^2 - b_j^2) \equiv \sum_j (a_j - b_j) \\
& \equiv \sum_j \left((a_j - b_j) + (1 + i)b_j\right) \equiv \sum_j (a_j + b_ji)B^j \\
& \equiv a + bi \pmod{(1 + i)\G}.
\end{align*}
Now, if $a + bi$ is a Gaussian $B$-happy number, then for some $k\in \Z^+$, 
$S^k(a + bi) = 1$.
Thus, using an inductive argument, each Gaussian $B$-happy number is congruent to 1 modulo $(1 + i)\G$ and so is an element of $1 + (1+i)\G$.
\end{proof}

The converse of Theorem~\ref{coset} is certainly false: By Theorem~\ref{fixedeg}, for $B$ odd, $(B^2 + 1)/2$ is a fixed point of $S_B$ and, hence, is not a Gaussian $B$-happy number.  Yet, for $B$ odd, $B^2 + 1 \equiv 2 \pmod{4\Z}$ implying that $(B^2 + 1)/2 \equiv 1 \pmod{2\Z}$.  Hence, $(B^2 + 1)/2 \in 1 + 2\Z \subseteq 1 + (1+i)\G$. 

As a corollary to Theorem~\ref{coset}, we see that if $B$ is odd, then each Gaussian $B$-happy number has real and imaginary parts of different parity.

\begin{cor}
Let $B \geq 3$ be odd.  If $a + bi$ is a Gaussian $B$-happy number, then $a + b \equiv 1 \pmod{2\Z}$.
\end{cor}

\begin{proof}
Let $a + bi$ be a Gaussian $B$-happy number.  By Theorem~\ref{coset}, $a + bi - 1 \in (1+i)\G$.  Since $b - bi \in (1+i)\G$, this implies that $a + b - 1 \in (1+i)\G$,
and hence $a + b - 1 \in \Z \cap (1+i)\G  = 2\Z$.
\end{proof}


Before proving a generalization of Theorem~\ref{GT} to Gaussian $B$-happy numbers, we
note the equivalence of the existence of $D$-consecutive sequences of Gaussian $B$-happy numbers for some related values of $D$.  The proof follows easily from Lemma~\ref{rep}.

\begin{lemma}\label{otherds}
If there exists a $D$-consecutive sequence of Gaussian $B$-happy numbers for some $D\in \G - \{0\}$, then there exists a $D^\prime$-consecutive sequence of Gaussian $B$-happy numbers of the same length, for $D^\prime$ equal to each of $-D$, $\pm iD$, $\pm \overline D$, and $\pm i\overline D$.
\end{lemma}

We now generalize Theorem~\ref{GT}.  
 
\begin{theorem} \label{easyseq}
Fix $B \geq 2$ and let $d = \gcd(2,B-1)$.
There exist arbitrarily long finite
$d$-consecutive sequences of Gaussian $B$-happy numbers, and  
$d$ is the smallest element of $\Z^+$ for which this is true.
The same holds for $-d$-consecutive, $id$-consecutive, and $-id$-consecutive sequences.
\end{theorem}
 
\begin{proof}
The existence of the sequences is immediate from Theorem~\ref{GT}, since rational $B$-happy numbers are also Gaussian $B$-happy numbers. For $B$ even, $d = 1$, which is clearly the minimal value possible.  For $B$ odd, Theorem~\ref{coset} eliminates the possibility of $d = 1$.  Thus the result is best possible in each case.  Lemma~\ref{otherds} proves the final sentence of the theorem.
\end{proof}




Theorem~\ref{1+i}, which holds for all $B \geq 2$, establishes that there exist arbitrarily long finite $(1+i)$-consecutive sequences of Gaussian $B$-happy numbers.  For its proof, we need to define a function that serves as a one-sided inverse for $S_B$.

Fix $B \geq 2$.  We define a function, $R_B: \Z^+ \rightarrow \Z^+$ by, for each $t \in \Z^+$,
\[R_B(t) = \sum_{j=1}^t B^j.\]
Notice that, for each $t\in \Z^+$, 
\begin{equation*}\label{RB}
S_B(R_B(t)) = t.
\end{equation*}

\begin{theorem}\label{1+i}
For $B \geq 2$ and $D = 1 + i$, 
there exist arbitrarily long finite
$D$-consecutive sequences of Gaussian $B$-happy numbers.
\end{theorem}

\begin{proof}
Let $m\in \Z^+$ be arbitrary.  We will show that there exists a $D$-consecutive sequence of $m$ Gaussian $B$-happy numbers.

Let $d = \gcd(2,B-1)$ and
\[M = \max\{S_B(2S_B(k))|1\leq k \leq m\}.\]  
By Theorem~\ref{GT}, there exists 
a sequence of $M$ $d$-consecutive rational $B$-happy numbers, say $a + dj$, for $0 \leq j < M$.
Set $r = 1 + \max\{k,2S_B(k)|1\leq k\leq m\}$. (Note that this means that $r$ is certainly large enough for the application of Lemma~\ref{pullapart} in the following calculation.)

Let $b = R_B(R_B(a+dM)B^r)$.  Then for each $1\leq k \leq m$,
\begin{align*}
S_B^2(bB^r + k(1+i)) &= S_B(S_B(b) + S_B(k(1+i))) \\
& = S_B(S_B(R_B(R_B(a+dM)B^r)) + 2S_B(k)i) \\
& = S_B(R_B(a+dM)B^r + 2S_B(k)i) \\
&= S_B(R_B(a+dM)) + S_B(2S_B(k)i) \\
& = a + dM - S_B(2S_B(k)).
\end{align*}

By the definition of $M$, for each $k$, $1 \leq S_B(2S_B(k)) \leq M$ and, therefore, $a \leq a + dM - S_B(2S_B(k)) < a + dM$.  So, if $d = 1$,
then $a + dM - S_B(2S_B(k))$ is in the sequence of $M$ $d$-consecutive rational $B$-happy numbers.
If $d = 2$, then $B$ is odd, and $S_B(2S_B(k))$ is even.  Hence, 
$a + dM - S_B(2S_B(k))$ is again in the sequence of $M$ $d$-consecutive rational $B$-happy numbers.
Thus, in either case, 
for each $k$, 
$a + dM - S_B(2S_B(k))$ is a $B$-happy number.  
Therefore, for each $1\leq k \leq m$, $bB^r + k(1+i)$ is a Gaussian $B$-happy number, and so these numbers form a $D$-consecutive sequence of $m$ Gaussian $B$-happy numbers.
\end{proof}

Combining Lemma~\ref{otherds} with Theorem~\ref{1+i} yields the corollary.

\begin{cor}
Let $B \geq 2$.
For $D = 1 - i$, $-1 + i$, and $-1 -i$,
there exist arbitrarily long finite
$D$-consecutive sequences of Gaussian $B$-happy numbers.
\end{cor}

\section{Acknowledgements}

This work was supported by the Research Experiences for Undergraduate Faculty (REUF) program. REUF is a program of the American Institute of Mathematics (AIM) and the Institute for Computational and Experimental Mathematics (ICERM), made possible by the support from the National Science Foundation (NSF) through DMS 1620073 to AIM and 1620080 to ICERM.  At ICERM, Brown University provided further support through the use of the facilities of its Center for Computation and Visualization.


\end{document}